\documentclass{amsart}

\usepackage[english]{babel}

\usepackage[a4paper,top=2cm,bottom=2cm,left=3cm,right=3cm,marginparwidth=1.75cm]{geometry}

\usepackage{amsmath, amssymb}
\usepackage{graphicx}
\usepackage[colorlinks=true, allcolors=blue]{hyperref}
\usepackage{mathtools}
\usepackage{amsfonts}
\usepackage{amsthm}
\usepackage{tikz-cd}
\usepackage{polski}
\usepackage{comment}
\usepackage{caption}
\newtheorem{thm}{Theorem}[section]

\newtheorem{cor}[thm]{Corollary}
\newtheorem{lem}[thm]{Lemma}
\newtheorem{prop}[thm]{Proposition}
\newtheorem{defn}[thm]{Definition}
\newtheorem{correspondence}[thm]{Correspondence}
\theoremstyle{remark}
\newtheorem{rem}[thm]{Remark}

\numberwithin{equation}{section}
\newcommand{\mycomment}[1]{}

\newcommand{\ZZ}{\mathbb Z}

\newcommand{\cA}{\mathcal{A}}

\newcommand{\cJ}{\mathcal{J}}

\newcommand{\tC}{\Tilde{C}}

\newcommand{\cE}{\mathcal{E}}

\newcommand{\cO}{\mathcal{O}}

\newcommand{\cR}{\mathcal{R}}

\newcommand{\s}{\sigma}
\newcommand{\st}{\sigma\tau}

\newcommand{\unnumberedfootnote}[1]{%
  \begingroup
  \renewcommand{\thefootnote}{}%
  \footnotetext[0]{#1}%
  \endgroup
}

\DeclareMathOperator{\Aut}{Aut}
\DeclareMathOperator{\Gal}{Gal}
\DeclareMathOperator{\id}{id}
\DeclareMathOperator{\Id}{Id}
\DeclareMathOperator{\Fix}{Fix}
\DeclareMathOperator{\im}{Im}
\DeclareMathOperator{\Nm}{Nm}
\DeclareMathOperator{\Pic}{Pic}

\DeclareMathOperator{\End}{End}

\title{Pryms of $\mathbb{Z}_3\times\mathbb{Z}_3$ coverings of genus 2 curves}
\author{Pawe\l{} Bor\'owka, Anatoli Shatsila}
\begin{document}
\begin{abstract}
We study unramified Galois $\mathbb{Z}_3 \times \mathbb{Z}_3$-coverings of genus 2 curves and the corresponding Prym varieties and Prym maps. In particular, we prove that any such covering can be reconstructed from its Prym variety, that is, the Prym--Torelli theorem holds for these coverings. We also investigate the Prym map of unramified $G$-coverings of genus 2 curves for an arbitrary abelian group $G$. We show that the generic fiber of the Prym map is finite unless $G$ is cyclic of order less than 6.
\end{abstract}
\maketitle

\unnumberedfootnote{\textit{2020 Mathematics Subject Classification:} 14H40, 14H30, 14H45, 14K12}
\unnumberedfootnote{\textit{Key words and phrases:} Prym variety, Prym map, coverings of curves.}

\section{Introduction}

Given a finite morphism between smooth curves, one can canonically associate to it a polarised abelian variety, the Prym variety. This induces a map from the moduli space of coverings to the moduli space of polarised abelian varieties, known as the Prym map. Classically, the main aim of the Prym theory has been to understand the Prym map for unramified double coverings. However, more recently, other types of coverings have gained a lot of interest. Before going into details, let us introduce some notation.

Let $f:C'\to C$ be a covering of smooth curves of genera $g'$ and $g$. It induces the so-called Norm map $Nm_f:JC'\to JC$ that is defined as a linear extension of the formula $Nm_f(P)=f(P)$ (where we identify the Jacobian variety with the Picard group of degree \(0\) divisors on a curve). Then, the Prym variety $P(f) = P(C'/C)$ of a covering $f$ can be defined as an abelian subvariety that is the connected component of the kernel of the Norm map. It carries a natural polarisation given by the restriction of the principal polarisation on $JC'$. The Prym map assigns to a covering its (polarised) Prym variety. In some cases, the geometry of a Prym map is quite surprising. In particular, there are Prym maps of degree $27, 16, 10$ and $3$, see the survey \cite{Oams} and references therein. On the other hand, not much is known for many types of coverings. Even if we assume that the genus of the bottom curve $2$ and the covering is unramified, we have only partial results. Recently, the generic injectivity of the Prym map of unramified cyclic coverings of genus 2 curves of degrees equal to Sophie Germain primes was proved in \cite{sophie}. This result was generalised in the preprint \cite{naranjo2024simplicityjacobiansautomorphisms} to other primes and genera.
In \cite{Ago}, there is a result stating that the Prym maps of unramified cyclic coverings of degree greater than 5 are generically finite. In \cite{BO19}, the authors have described $\ZZ_2\times\ZZ_2$ coverings and showed the global injectivity of the Prym map in this case.

Following the results mentioned above, we have decided to investigate more closely abelian coverings of genus \(2\) curves, i.e. Galois coverings $f:\tC\to C$ with the Galois group being an abelian group. By \cite{Pardini}, any such covering can be constructed from a curve $C$ and a subgroup of its Jacobian $JC$. The Prym maps of such coverings have been investigated in \cite{Mohajer}, however, the results obtained there cannot be applied to genus \(2\) curves. 

The first part of the manuscript is devoted to the study of $\mathbb{Z}_3\times \mathbb{Z}_3$-coverings of genus two curves and their Prym maps. 
Similarly to \cite{BO19}, the moduli space $\mathcal{R}_2^{\mathbb{Z}_3\times \mathbb{Z}_3}$ of $\mathbb{Z}_3\times \mathbb{Z}_3$-coverings consists of two irreducible components: the component $(\mathcal{R}_2^{\mathbb{Z}_3\times \mathbb{Z}_3})^{iso}$ of isotropic coverings and $(\mathcal{R}_2^{\mathbb{Z}_3\times \mathbb{Z}_3})^{ni}$ of non-isotropic coverings, see Proposition \ref{moduli} and Definition \ref{defniso}. 
The main result of the paper is the global injectivity of the Prym map on both components:
\begin{thm}[Theorem \ref{mainthm}]
    The Prym maps $$\mathcal{P}^{iso}:(\mathcal{R}_2^{\mathbb{Z}_3\times \mathbb{Z}_3})^{iso} \to \cA^{(1,1,1,1,1,1,3,3)}_8$$ 
    $$\mathcal{P}^{ni}:(\mathcal{R}_2^{\mathbb{Z}_3\times \mathbb{Z}_3})^{ni} \to\cA^{(1,1,1,1,1,1,1,9)}_8$$
    are injective.
\end{thm}

Since the proof is quite technical and has a few steps, let us start with an overview of the proof.
Let $[f: \tC \to H] \in \mathcal{R}_2^{\mathbb{Z}_3\times \mathbb{Z}_3}$ be a covering with $H = \tC/\langle \sigma, \tau \rangle$ and $(P(f),\Theta_{P(f)})$ be its polarised Prym variety. First, we show that the automorphisms of $P(f)$ induced by $\langle \sigma, \tau\rangle$ can be recovered from the polarisation $\Theta_{P(f)}$. Then the action of $\langle \sigma, \tau \rangle$ on the Prym variety induces an isotypical decomposition of $P(f)$ as a product of four Prym varieties of lower dimension. The induced polarisations  uniquely determine three elliptic curves inside each of the four Prym varieties. These twelve elliptic curves are, in fact, quotients of $\tC$ by the subgroups of $\Aut(\tC)$ isomorphic to $S_3$ and generated by an element in $\langle \sigma, \tau\rangle$ and a lift of the hyperelliptic involution $\iota$ on $H$ to $\tC$. 

Having distinguished elliptic curves, we use them to recover the set of involutions of the Prym variety $P(f)$ induced by the aforementioned lifts. Each such involution corresponds to a $4$-tuple of elliptic curves, on which it acts as the identity, and
every such $4$-tuple generates a Jacobian of a genus four curve, which is the quotient of $\tC$ by one of the nine lifts of $\iota_H$. Thus, we recover the coverings of elliptic curves by genus four curves. The coverings are non-Galois and their Galois closures give the coverings of elliptic curves by $\tC$, and determine the covering $f: \tC \to H$ uniquely. 

The approach above works in both isotropic and non-isotropic cases. However, the isotropic case requires more careful treatment due to the fact that the small Pryms are not embedded in the Prym variety $P(f)$. In particular, one has to pass to the quotient of each small Prym by a carefully chosen $3$-torsion point to conclude that elliptic curves are uniquely determined by the polarisation (see Proposition \ref{eliso}). 

It is important to note that the results of this paper cannot be obtained from known results. The methods of \cite{BO19} rely on an explicit description of hyperelliptic curves via Weierstrass points. The methods of \cite{sophie, naranjo2024simplicityjacobiansautomorphisms} use the genericity assumption in many ways, so one cannot hope for results like global injectivity that we have been able to obtain.
As an application of the first part, in the second part of the paper we focus on the general structure of the Prym map for abelian Galois coverings of genus \(2\). We show that the Prym map is generically finite apart from a few exceptions. The exact result is as follows. 
\begin{thm}[Theorem \ref{thmabel}]
    Let $G$ be a finite abelian group.  
    \begin{enumerate}
        \item If $G$ is isomorphic to either $\mathbb{Z}_2, \mathbb{Z}_3, \mathbb{Z}_4$ or $\mathbb{Z}_5$, then the generic fiber of the Prym map of unramified $G$-coverings of genus 2 curves $$\mathcal{P}_2^{G}: \mathcal{R}_2^{G} \to \mathcal{A}^{G}$$ is positive dimensional. 
        \item If $G$ is not isomorphic to a cyclic group of order less than 6, then the Prym map $$\mathcal{P}_2^{G}: \mathcal{R}_2^{G} \to \mathcal{A}^{G}$$ is generically finite. 
    \end{enumerate}
    Note that when the rank of $G$ (i.e. the minimal number of generators) is at least 5 then $\mathcal{R}_2^{G}$ is empty.
\end{thm}

\subsection*{Acknowledgements}
The results on a non-isotropic case of the second author has been supported by the Polish National Science Center project number 2019/35/D/ST1/02385. The results on an isotropic case of the first author has been supported by the Polish National Science Center project number 2024/54/E/ST1/00330.
The authors would like to thank Juan Carlos Naranjo for some stimulating discussions during his stay at the Jagiellonian University in Kraków. His stay was funded by the program Excellence Initiative –
Research University at the Jagiellonian University in Kraków. 
The authors would like to thank anonymous reviewers for their careful reading and comments that improved the paper.

\section{Preliminaries}
In this section, we introduce the notation and prove results about the Prym varieties of unramified $\mathbb{Z}_3 \times \mathbb{Z}_3$-coverings of genus 2 curves.

We would like to start by computing the number of irreducible components of the moduli of coverings. As stated in the introduction, we define the moduli space of $\ZZ_3\times\ZZ_3$-coverings as 
$$\mathcal{R}_2^{\mathbb{Z}_3\times \mathbb{Z}_3}=\{(H,\left<\eta,
\zeta\right>): H \text{ is a smooth genus 2 curve}, \eta,\zeta\in JH[3]\, \left<\eta,\zeta\right>\simeq\ZZ_3\times\ZZ_3\}$$
where $JH[3]$ is the set of 3-torsion points of the Jacobian $JH$, see \cite{Pardini}.

We also can consider the moduli of principally polarised abelian surfaces with level 3--structure, see \cite[Chap 8.3.2]{BL}, i.e. the space $$\cA_2[3]=\{(A,e_1,f_1,e_2,f_2): A\in\cA_2,\ e_1,f_1,e_2,f_2\in A[3] \text{ is a symplectic basis for } \omega\},$$
where $\omega$ is the symplectic form on the set of 3-torsion points $A[3]$, usually called the Weil form or the Commutator map (see \cite[Chap 6.3]{BL}).

Let $\cJ_2[3]\subset \cA_2[3]$ be an open dense subset of Jacobians of smooth genus 2 curves.
The first result of the paper is the following proposition.
\begin{prop}\label{moduli}
    The moduli space $\mathcal{R}_2^{\mathbb{Z}_3\times \mathbb{Z}_3}$ contains precisely two disjoint irreducible components, $(\mathcal{R}_2^{\mathbb{Z}_3\times \mathbb{Z}_3})^{iso}$ and $(\mathcal{R}_2^{\mathbb{Z}_3\times \mathbb{Z}_3})^{ni}$ called isotropic and non-isotropic $\mathbb{Z}_3\times \mathbb{Z}_3$ coverings.
\end{prop}
\begin{proof}
Let $(H,\left<\eta,\zeta\right>) \in \mathcal{R}_2^{\mathbb{Z}_3\times \mathbb{Z}_3}$ be a covering. Consider $(JH[3], \omega)$ that is a symplectic vector space over the field of 3 elements. If $\omega(\eta,\zeta)\neq0$ then $e_1 :=\eta, f_1 :=\omega(\eta,\zeta)^{-1}\zeta$ can be extended to a symplectic basis $e_1,e_2,f_1,f_2$. If $\omega(\eta,\zeta)=0$ then $\left<\eta,\zeta\right>$ is a maximal isotropic subgroup. We can take $e_1 :=\eta, e_2 :=\zeta$ and then find $f_1\in e_2^{\perp}\setminus\left<e_1,e_2\right>$ such that $\omega(e_1,f_1)=1$ and $f_2\in \left<e_1,f_1\right>^{\perp}$ such that $\omega(e_2,f_2)=1$. Therefore, the set $\{\eta, \zeta\}$ can always be extended to a symplectic basis of $(JH[3], \omega)$.

    Consider the following two forgetful maps
    \begin{align*}
    \pi_{iso}&:\cJ_2[3]\to (\mathcal{R}_2^{\mathbb{Z}_3\times \mathbb{Z}_3})^{iso}, \ &\pi_{iso}(JH,e_1,e_2,f_1,f_2)=(H,\left<e_1,e_2\right>)\\
    \pi_{ni}&:\cJ_2[3]\to (\mathcal{R}_2^{\mathbb{Z}_3\times \mathbb{Z}_3})^{ni}, \ &\pi_{ni}(JH,e_1,e_2,f_1,f_2)=(H,\left<e_1,f_1\right>)    
    \end{align*}
        By what we have written before, the maps are well-defined and surjective. Since the domain is irreducible, both spaces are irreducible.
\end{proof}

Now, we describe the geometry of the coverings. We start with the 'from top-down' perspective. 
Let $f: \tC \to H$ be an unramified Galois covering of a smooth genus 2 curve $H$ with a Galois group $G = \mathbb{Z}_3 \times \mathbb{Z}_3 = \langle \sigma, \tau \rangle \subset \Aut(\tC)$. The curve $\tC$ has genus 10 and there are four genus 4 curves which are quotients of $\tC$ by subgroups of $G$ isomorphic to $\mathbb{Z}_3$. We denote the quotient curve $\tC/\left<\alpha\right>$ by $C_\alpha$ and the quotient maps $\tC \to C_{\alpha}$ and $C_{\alpha} \to H$ by $k_\alpha$ and $h_{\alpha}$, respectively. We have the following diagram:

\begin{figure}[h!]
\begin{tikzcd}
                                       &                                  & \tC \arrow[dddd, "f"] \arrow[ldd, "k_{\tau}"] \arrow[lldd, "k_{\sigma}"'] \arrow[rdd, "k_{\sigma\tau}"'] \arrow[rrdd, "k_{\sigma^2\tau}"] &                                               &                                                   \\
                                       &                                  &                                                                                                                                           &                                               &                                                   \\
C_{\sigma} \arrow[rrdd, "h_{\sigma}"'] & C_{\tau} \arrow[rdd, "h_{\tau}"] &                                                                                                                                           & C_{\sigma\tau} \arrow[ldd, "h_{\sigma\tau}"'] & C_{\sigma^2\tau} \arrow[lldd, "h_{\sigma^2\tau}"] \\
                                       &                                  &                                                                                                                                           &                                               &                                                   \\
                                       &                                  & H                                                                                                                                         &                                               &                                                  
\end{tikzcd}
\caption{}
\end{figure}

It follows from \cite{ries} that the hyperelliptic involution $\iota \in \Aut(H)$ lifts to an involution $j \in \Aut(\Tilde{C})$ so that $\sigma, \tau, j$ generate the subgroup $F \subseteq \Aut(\Tilde{C})$, isomorphic to SmallGroup(18,4) in the GAP Small Groups Library (see \cite{GAP}), and given by $$F = \langle \sigma, \tau, j \: | \: \sigma^3 = \tau^3 = j^2 = (j\sigma)^2 = (j\tau)^2 = 1, \sigma\tau = \tau\sigma \rangle.$$ 

\begin{lem}
\label{conj}
All involutions in $F$ are conjugate.
\end{lem}
\begin{proof}

    It is enough to show that all the involutions in $F$ are conjugate to $j$. Note that all elements of $F$ can be written in the form $j^c\sigma^a\tau^b$ for $0 \leq c \leq 1$ and $0 \leq a,b \leq 2$ and $j^c\sigma^a\tau^b$ is of order two if and only if $c = 1$. Let $j\sigma^a\tau^b = jg$ be any involution in $F$. Since $jg^njg^n = 1$ for any positive $n$, we get $jg^{2n} = g^{-n}jg^{n}$. Taking $n= 2$ we see that $jg$ and $j$ are conjugate.
    
\end{proof}

For the simplicity of exposition, we will now focus on the curve $C_\s$ only. The results concerning other quotients follow from the relabelling of order 3 automorphisms in $F$. 

The involutions $j, j\sigma, j\sigma^2\in \Aut(\tC)$ induce an involution on $C_{\sigma}$, which, for simplicity, we also denote by $j$. Note that this involution is also one of the lifts of the hyperelliptic involution $\iota\in \Aut(H)$ by $h_{\sigma}$. 

For $0 \leq a,b \leq 2$, let $C_{ab} := \Tilde{C}/j\sigma^a\tau^b$ and $\pi_{ab}: \Tilde{C} \to C_{ab}$ be the double covering. Lemma \ref{conj} implies that all curves $C_{ab}$ are pairwise isomorphic. Moreover, since \(f\) is unramified, it follows that  each $\pi_{ab}$ is ramified at 6 points, one in each preimage by $f$ of the Weierstrass points of $H$. Then, the Riemann-Hurwitz formula implies that $g(C_{ab}) = 4$ for each $0 \leq a,b \leq 2$. Note that there are 12 subgroups of $F$ isomorphic to $S_{3}$. The quotients of $\tC$ by each of these subgroups are elliptic curves and with notation $E_{\sigma, j\tau} = \tC /\langle \sigma, j\tau \rangle$ we have Diagram \ref{bigdiag}, where, for brevity, the complete information about the coverings is given only for $C_{\sigma}$ and $C_{00}$:

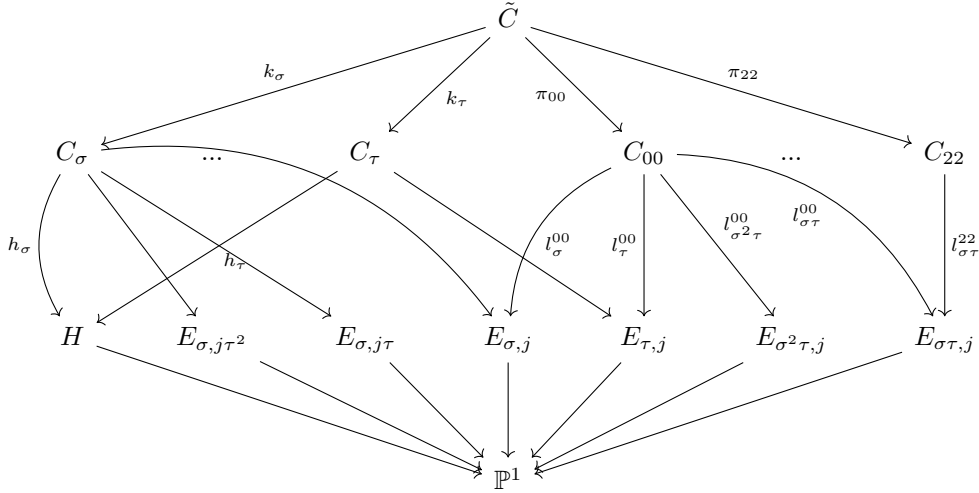
\begin{figure}
\begin{tikzcd}                                                                                                       &                                    &                                                  & \tC \arrow[rdd, "\pi_{00}"'] \arrow[rrrdd, "\pi_{22}"] \arrow[ldd, "k_{\tau}"] \arrow[llldd, "k_{\sigma}"'] &                                                                                                                                                                         &                                    &                                           \\
                                                                                                       &                                    &                                                  &                                                                                                             &                                                                                                                                                                         &                                    &                                           \\
C_{\sigma} \arrow[ddd, "h_{\sigma}"', bend right] \arrow[rddd] \arrow[rrddd] \arrow[rrrddd, bend left] & ...                                & C_{\tau} \arrow[rrddd] \arrow[llddd, "h_{\tau}"] &                                                                                                             & C_{00} \arrow[ddd, "l^{00}_{\tau}"'] \arrow[lddd, "l^{00}_{\sigma}", bend right] \arrow[rrddd, "l_{\sigma\tau}^{00}"', bend left] \arrow[rddd, "l^{00}_{\sigma^2\tau}"] & ...                                & C_{22} \arrow[ddd, "l^{22}_{\sigma\tau}"] \\
                                                                                                       &                                    &                                                  &                                                                                                             &                                                                                                                                                                         &                                    &                                           \\
                                                                                                       &                                    &                                                  &                                                                                                             &                                                                                                                                                                         &                                    &                                           \\
H \arrow[rrrdd]                                                                                        & {E_{\sigma, j\tau^2}} \arrow[rrdd] & {E_{\sigma, j\tau}} \arrow[rdd]                  & {E_{\sigma, j}} \arrow[dd]                                                                                  & {E_{\tau, j}} \arrow[ldd]                                                                                                                                               & {E_{\sigma^2\tau, j}} \arrow[lldd] & {E_{\sigma\tau, j}} \arrow[llldd]         \\
                                                                                                       &                                    &                                                  &                                                                                                             &                                                                                                                                                                         &                                    &                                           \\
                                                                                                       &                                    &                                                  & \mathbb{P}^1                                                                                                &                                                                                                                                                                         &                                    &                                          
\end{tikzcd}
\caption{}
\label{bigdiag}
\end{figure}

From Diagram \ref{bigdiag} we see that $C_{\sigma}$ is a double covering of three elliptic curves, one of which is covered by $C_{00}$ but none is covered by $C_{\tau}, C_{\sigma\tau}, C_{\sigma^2\tau}$. The curve $C_{00}$ is a (non-Galois) triple covering of four elliptic curves, and for any other curve $C_{ab}$ there is a unique elliptic curve covered by both $C_{00}$ and $C_{ab}$. We denote the covering $C_{00} \to E_{\s, j}$ by $l_{\s}^{00}$ and analogously for other such coverings. 

Recall that the Prym variety of the covering $f: \tC \to H$ is defined by $$P = P(f) := \ker(\Nm_f)^0.$$ In the rest of the section we will focus on the geometry of the Prym variety $P$. Note that $P$ comes with the restricted polarisation from $J\tC$ that we will call $\Theta_P$. 

\begin{lem}
    The pullback of the Jacobian $JC_{ab}$ via $\pi_{ab}$ lies in $P$ for any $a,b\in\{0,1,2\}$.
\end{lem}

\begin{proof}
    Since the covering $\pi_{ab}$ is of degree 2 and ramified, $\pi_{ab}^*: JC_{ab} \to J\Tilde{C}$ is injective \cite[Proposition 11.4.3]{BL} and $\pi_{ab}^*(JC_{ab}) = \im (1 + j\sigma^a\tau^b) \subset J\Tilde{C}$.
    Since the Jacobians are connected, it is enough to show that $ \im (1 + j\sigma^a\tau^b) \subset \ker(\sum_{0 \leq a,b \leq 2} \sigma^a\tau^b)$. But this follows from
    $$\left(\sum_{0 \leq a,b \leq 2} \sigma^a\tau^b\right)(1 + j\sigma^a\tau^b) = \sum_{\substack{0 \leq c \leq 1,\\0 \leq a,b \leq 2}}j^c\sigma^a\tau^b = 0.$$
\end{proof}

Up to this point, we have been looking at the construction from the 'top-down' perspective. Now, we would like to look from the bottom-up one. Recall that the construction is obtained from $(H,\left<\eta,\zeta\right>)$, where $\eta,\zeta\in JH[3]$. From now on, we will assume that the covering $h_{\s}$, that is, the quotient map of $C_{\s}$ by $\langle \tau \rangle$, is given by $\eta\in JH[3]$ and the covering $k_{\s}$ is given by $h_{\s}^*(\zeta) \in JC_{\s}[3]$ with $\zeta \in JH[3]$. Then $\langle \eta, \zeta \rangle$ corresponds to $f$ and $\ker f^* = \langle \eta, \zeta \rangle \subset JH[3]$ is isomorphic to $\mathbb{Z}_3 \times \mathbb{Z}_3$. The following definition follows directly from Proposition \ref{moduli}.
\begin{defn}\label{defniso}
A $\ZZ_3\times\ZZ_3$ covering $f:\tC\to H$ given by $\langle \eta, \zeta \rangle \subset JH[3]$ is called \it{isotropic} if $e^{3\Theta_H}(\eta, \zeta) = 1$ and \it{non-isotropic} if $e^{3\Theta_H}(\eta, \zeta) \neq 1$, where $e^{3\Theta_H}$ is the usual Weil pairing (Commutator map) on $JH[3]$, see \cite[Chap 6.3]{BL}.
\end{defn}

\begin{lem}
Let $P$ be the Prym variety of $f:\tC\to H$ with the restricted polarisation called $\Theta_P$.
\label{polariz}
\begin{itemize}
    \item[i)] If the covering $f$ is isotropic, then $\Theta_P$ has type $(1,1,\ldots, 1, 3,3)$;
    \item[ii)] If the covering $f$ is non-isotropic, then $\Theta_P$ has type $(1,1,\ldots, 1, 9)$.
\end{itemize}
    
\end{lem}

\begin{proof}
    By \cite[Proposition 3.2.9]{LangeRodr} we have $$K(\Theta_P) \simeq (\ker f^*)^{\bot}/ \ker f^*.$$
    Let $e_1, f_1, e_2, f_2$ be  symplectic generators of $JH[9]$ (with $e^{9\Theta_H}(e_i, f_i) = e^{\frac{2\pi i }{9}}$ for $i = 1,2$). By a possible change of coordinates, we can assume that in the isotropic case we have $\ker f^* = \langle 3e_1, 3e_2 \rangle$. Then $(\ker f^*)^{\bot} = \langle e_1, 3f_1, e_2, 3f_2 \rangle$. Therefore, $K(\Theta_P) \simeq\mathbb{Z}_3^2 \times \mathbb{Z}_3^2 $, hence the type is $(1,1,\ldots, 1,3,3)$. 
    
    In the non-isotropic case we can assume that $\ker f^* = \langle 3e_1, 3f_1 \rangle$, and hence $(\ker f^*)^{\bot} = \langle 3e_1, 3f_1, e_2, f_2\rangle$. Thus, $K(\Theta_P) \simeq \mathbb{Z}_{9}^2$, and the type is $(1,1,\ldots, 1,9)$.
\end{proof}

Before determining the isotypical decomposition of the Jacobian $J\tC$ we introduce the following notation. Let $A$ be an abelian variety and $A_1, A_2 \ldots, A_n \subseteq A$ be a set of its abelian subvarieties with the induced polarisations. We denote the symmetric idempotent of $A_k$ by $\varepsilon_{k} \in \End_{\mathbb{Q}}(A)$ for $1 \leq k \leq n$. We write $$A = A_1 \boxplus A_2 \boxplus \ldots \boxplus A_n$$ if $$\sum_{i=1}^n \varepsilon_i = 1.$$ 

\begin{prop}
\label{rozklad}
    We have the following decompositions:
    $$J\tC = f^*JH \boxplus k_{\sigma}^*P(h_{\sigma}) \boxplus k_{\tau}^*P(h_{\tau}) \boxplus k_{\sigma\tau}^*P(h_{\sigma\tau}) \boxplus k_{\sigma^2\tau}^*P(h_{\sigma^2\tau}).$$
    With the notation from diagram \ref{bigdiag} we have
    $$JC_{00} = (l_{\sigma}^{00})^*E_{\sigma, j} \boxplus (l_{\tau}^{00})^*E_{\tau, j} \boxplus (l_{\sigma\tau}^{00})^*E_{\sigma\tau, j} \boxplus (l_{\sigma^2\tau}^{00})^*E_{\sigma^2\tau, j}.$$
\end{prop}

\begin{proof}
    
    The symmetric idempotent of $f^*JH$ is given by $$\varepsilon_{f^*JH} = \frac{\sum_{0\leq a,b \leq 2}\sigma^a\tau^b}{9}.$$ Note that the idempotent of $k_{\sigma}^*P(h_{\sigma})$ is given by $$\varepsilon_{k_{\sigma}^*P(h_{\sigma})} = \frac{(1 + \sigma + \sigma^2)(2 - \tau - \tau^2)}{9}.$$
    Therefore, we get $$\varepsilon_{f^*JH} + \varepsilon_{k_{\sigma}^*P(h_{\sigma})} + \varepsilon_{k_{\tau}^*P(h_{\tau})} + \varepsilon_{k_{\sigma\tau}^*P(h_{\sigma\tau})} + \varepsilon_{k_{\sigma^2\tau}^*P(h_{\sigma^2\tau})} = 1.$$

    Now, note that the morphism $1 + \sigma + \sigma^2$ induces a well-defined morphism of $JC_{00}$ since it commutes with $(1 + j)$. In fact, this is precisely the norm endomorphism of $(l_{\sigma}^{00})^*JE_{\sigma,j}$.
    Thus, we get $$\varepsilon_{(l_{\tau}^{00})^*JE_{\tau, j}} + \varepsilon_{(l_{\sigma}^{00})^*JE_{\sigma, j}} + \varepsilon_{(l_{\sigma\tau}^{00})^*JE_{\sigma\tau, j}} + \varepsilon_{(l_{\sigma^2\tau}^{00})^*JE_{\sigma^2\tau, j}} = \frac{3 + \sum_{0 \leq a,b \leq 2}\sigma^a\tau^b}{3} = 1,$$ since on $JC_{00}$ we have $$\sum_{0\leq a,b \leq 2}\sigma^a\tau^b = 0.$$ 
\end{proof}

\begin{cor}
\label{decprym}
    The isotypical decomposition of $P=P(f)$ is as follows:
    $$P = k_{\sigma}^*P(h_{\sigma}) \boxplus k_{\tau}^*P(h_{\tau}) \boxplus k_{\sigma\tau}^*P(h_{\sigma\tau}) \boxplus  k_{\sigma^2\tau}^*P(h_{\sigma^2\tau}).$$
\end{cor}
From now on, we will call isotypical components of $P(f)$ \textit{small Pryms}. 
As before, we will focus on the first component $k_{\sigma}^*P(h_{\sigma})$, denoted by $P_\s$ for simplicity.

\begin{rem}
\label{prymisog}
    It follows from \cite[Proposition 2.5]{ortega} or \cite[Theorem 1, page 67]{ries} that the Prym variety of the covering $h_{\sigma}:C_{\sigma} \to H$ is abstractly isomorphic to $E_{\sigma, j} \times E_{\sigma, j\tau}$. Hence, using Proposition \ref{rozklad} and Corollary \ref{decprym} we can deduce that the Prym variety $P$ is abstractly isogenous to $JC_{00} \times JC_{10}$.
\end{rem}

Now, we would like to compute the restricted polarisation types of the small Pryms. We start with the following lemma.
\begin{lem}\label{trivialkernel}
\label{kernel}
    The (restricted) kernel $\ker k_{\sigma}^*|_{P(h_{\sigma})}$ is trivial in the non-isotropic case, and has cardinality $3$ in the isotropic case. In particular, in the non-isotropic case, we have $P_\s\simeq P(h_{\sigma})$.
\end{lem}

\begin{proof}
    Recall that $\ker f^* = \langle \eta, \zeta \rangle$, $\ker h_{\sigma}^* = \langle \eta \rangle$ and $\ker k_{\sigma}^* = \langle h_\sigma^*(\zeta) \rangle$. Since $$h_{\sigma}^*JH \cap P(h_{\sigma}) = h_{\sigma}^*((\ker h_{\sigma}^*)^{\bot})$$ it follows that in the non-isotropic case $h_{\s}^*(\zeta) \notin P(h_{\s})$ as $\zeta \notin (\ker h_{\s}^*)^{\bot}$.
    However, in the isotropic case we have $\zeta \in (\ker h_{\s}^*)^{\bot}$, so $h_{\sigma}^*(\zeta) \in P(h_{\sigma})$.
\end{proof}

\begin{lem}
\label{typeprym}
    The type of polarisation on $P_\s$ is $(3,9)$ in the non-isotropic case and $(3,3)$ in the isotropic case.
\end{lem}

\begin{proof}
Let us assume the covering is non-isotropic. 
Let $i: P(h_\s)\to JC_\s$ be an embedding of the small Prym into the Jacobian of the quotient curve. By Lemma \ref{trivialkernel}, the map $k_{\s}^*\circ i: P(h_\s)\to J\tC$ is an embedding (even though $k_{\s}^*$ is not), hence the restricted polarisation type on $P_\s$ can be computed from the type of the pullback polarisation on $JC_\s$ restricted to $P(h_\s)$. 

The restricted polarisation from $J\tC$ to $JC_\s$ is three times the principal polarisation because $k_\s$ is a triple covering.
The restricted polarisation from $JC_\s$ (seen as principally polarised) to $P(h_\s)$ is of type $(1,3)$ because $h_\s$ is a triple unramified cyclic covering.
Therefore, the restricted polarisation from $J\tC$ to $P_{\s}$ is of type $(3,9)$.

If the covering is isotropic then $k_\s^*|_{P(h_\s)}$ is not injective. Denoting by $\pi$ the quotient map $P\left(h_\sigma\right) \rightarrow P\left(h_\sigma\right) / h_\sigma^*(\zeta) \simeq P_{\sigma}$, we get the following diagram with the bottom arrow being embedding of $P_\s$ to $J\tC$:
\begin{equation}
\begin{tikzcd}[row sep=huge]
P(h_\s) \arrow[r,"i"] \arrow[d,swap,"\pi"] & JC_\s \arrow[d,swap,"k_\s^*"]
\\
P_\s \arrow[r,""] & J\tC 
\end{tikzcd}
\end{equation}
By the same argument as before, we know that the pullback polarisation from $J\tC$ to $P(h_\s)$ is of type $(3,9)$. Since $\pi$ is of degree 3, we have two possibilities for the type of restricted polarisation on $P_{\s}$, either $(3,3)$ or $(1,9)$. 

Now, note that $h_\sigma^*(\zeta)\in P(h_\s)\cap h^*(JC)$ lies in the kernel of $(1,3)$ polarisation on $P(h_\s)$. Thus, it follows from \cite{BL}[Lemma 2.4.7 c)] that there exists a primitive $9$-torsion point $z$ in the kernel of the $(3,9)$ polarisation on $P(h_\s)$ such that $3z = h_\sigma^*(\zeta)$. Therefore, the type of restricted polarisation on $P_{\s}$ is $(3,3)$.

\end{proof}

\begin{lem}
\label{degree}
     Let $$\Phi: f^*JH \times P_{\s} \times P_{\tau} \times P_{\s\tau} \times P_{\s^2\tau} \to J\tC$$ and $$\mu: P_{\s} \times P_{\tau} \times P_{\s\tau} \times P_{\s^2\tau} \to P(f)$$ be the addition maps. Then $$\deg \Phi = \begin{cases}
         3^{14}, \text{ if $f$ is non-isotropic}, \\
         3^{10}, \text{ if $f$ is isotropic.}
     \end{cases}$$ and $$\deg \mu = \begin{cases}
         3^{10}, \text{ if $f$ is non-isotropic}, \\
         3^{6}, \text{ if $f$ is isotropic.}
     \end{cases}$$
\end{lem}

\begin{proof}
     Assume that $f$ is non-isotropic. The Jacobian variety $(J\tC, \Tilde{\Theta})$ is principally polarised, and by Lemma \ref{typeprym} the product polarisation $\Xi = \Phi^*\Tilde{\Theta}$ on $f^*JH \times P_{\s} \times P_{\tau} \times P_{\s\tau} \times P_{\s^2\tau}$ is of type $(1,3,3,3,3,9,9,9,9,9)$. By \cite[Cor 3.6.6]{BL} we have $$\deg \Phi = \frac{\chi(\Xi)}{\chi(\Tilde{\Theta})} = \frac{3^4\cdot 9^5}{1} = 3^{14}.$$
     Other computations follow from Lemmas \ref{polariz} and \ref{typeprym} analogously. 
\end{proof}

\section{A proof of the main Theorem}

Our aim is to prove that the following Prym maps are globally injective $$\mathcal{P}^{iso}:(\mathcal{R}_2^{\mathbb{Z}_3\times \mathbb{Z}_3})^{iso} \to \cA^{(1,1,1,1,1,1,3,3)}_8$$ 
$$\mathcal{P}^{ni}:(\mathcal{R}_2^{\mathbb{Z}_3\times \mathbb{Z}_3})^{ni} \to\cA^{(1,1,1,1,1,1,1,9)}_8$$

In this section, when we do not want to distinguish isotropic and non-isotropic cases we denote the Prym map simply by $\mathcal{P}:\mathcal{R}_2^{\mathbb{Z}_3\times\mathbb{Z}_3} \to \mathcal{A}^{\mathbb{Z}_3\times\mathbb{Z}_3}$.

We start with the following proposition, where we show that the action of the Galois group $\mathbb{Z}_3 \times \mathbb{Z}_3$ on the Prym variety $P$ can be recovered from the restricted polarisation $\Theta_P$. The Proposition is analogous to \cite[Proposition 3.1]{sophie}. 

\begin{prop}
\label{group}
    Let $(P, \Theta_P)\in\im\mathcal{P}$. Then the subgroup of automorphisms $$N := \{ \psi \in \Aut(P,\Theta_P) \: | \: \psi \text{ has odd order and} \: \: \psi(x) = x \:\: \forall x \in K(\Theta_P)\}.$$
    is isomorphic to $\langle \sigma, \tau \rangle \simeq \mathbb{Z}_3 \times \mathbb{Z}_3$.
\end{prop}

\begin{proof}

    Let $f: \tC \to H$ be a covering with ${P}(f) = P$ and Galois group $\langle \sigma, \tau \rangle \simeq \mathbb{Z}_3 \times \mathbb{Z}_3$. Note that $K(\Theta_P) = f^*JH \cap P(f) \subset \Fix(\tau) \cap \Fix(\sigma)$ since the symmetric idempotent $$\varepsilon_{f^*JH} = \frac{\sum_{0\leq a,b \leq 2}\sigma^a\tau^b}{9}$$ of $f^*JH$ is fixed by both $\sigma, \tau \in \Aut(J\tC)$. Therefore, for induced automorphisms $\sigma, \tau \in \Aut(P, \Theta_P)$ we have $\langle \sigma, \tau \rangle \subseteq N$.  
    Let $\psi \in N$. Then, there is an automorphism $\Tilde{\psi}: J\tC \to J\tC$ such that the following diagram commutes:


\begin{equation}
        \begin{tikzcd}
0 \arrow[r] & K(\Theta_P) \arrow[d, equal] \arrow[r] & f^*JH \times P \arrow[d, "{(\id,\psi)}"] \arrow[r, "\mu"] & J\tC \arrow[d, "\Tilde{\psi}"] \arrow[r] & 0 \\
0 \arrow[r] & K(\Theta_P) \arrow[r]                                & f^*JH \times P \arrow[r, "\mu"]                           & J\tC \arrow[r]                           & 0
\end{tikzcd}
\end{equation}


where $\mu$ is the addition map. We want to show that $\Tilde{\psi}$ is polarised.

From the diagram we get that $\mu^*\Tilde{\psi}^*\mathcal{O}_{J\tC}(\Tilde{\Theta})$ and $ \mu^*\mathcal{O}_{J\tC}(\Tilde{\Theta})$ are equal as polarisations on $f^*JH \times P$ because $(\id,\psi)$ is a polarised isomorphism.
Now, since $\mu^*$ has a finite kernel, so  
$\Tilde{\psi}^*\mathcal{O}_{J\tC}(\Tilde{\Theta}) \otimes \mathcal{O}_{J\tC}(\Tilde{\Theta})^{-1}$ is a torsion sheaf, hence it belongs to $\Pic^0(J\tC)$. 
Therefore, $\Tilde{\psi}^*\mathcal{O}_{J\tC}(\Tilde{\Theta})$ induces the canonical polarisation $\Tilde{\Theta}$ on $J\tC$. 

Recall that \cite[Exercise 11.19]{BL} $$\Aut(\tC) = \begin{cases} \Aut(J\tC, \tilde{\Theta})  \text{ if } \tC \text{ is hyperelliptic} \\  \Aut(J\tC, \tilde{\Theta})/ \langle -1_{J\tC} \rangle \text{ if } \tC \text{ is nonhyperelliptic.} \end{cases}$$ Therefore, since $\Tilde{\psi}$ has odd order, there is an automorphism $\Tilde{\psi}_0$ of $\tC$ inducing $\Tilde{\psi}$. To prove that $\Tilde{\psi}_0 \in \Gal(\tC/H)$ we have to show that $\Tilde{\psi}_0$ descends to an identity morphism on $H$. Let us first show that $\Nm_f \circ (Id_{J\tC} - \Tilde{\psi}) = 0$. Writing $x \in J\tC$ as $y + f^*(z)$ with $y \in P, z \in JH$ we have $$(\Id_{J\tC} - \Tilde{\psi})(x) = y + f^*(z) - (\Tilde{\psi}(y) + f^*(z)) = y - \Tilde{\psi}(y).$$ Since $\Tilde{\psi}(y) \in P$ for any $y \in P$ we get $\Nm_f \circ (Id_{J\tC} - \Tilde{\psi}) = 0$. Let $p_1,p_2 \in \tC$ be two points in the same fiber of $f$. Then
$\Nm_f(p_1 - p_2 - \Tilde{\psi}_0(p_1) + \Tilde{\psi}_0(p_2)) = 0$, hence $f(\Tilde{\psi}_0(p_1)) = f(\Tilde{\psi}_0(p_2))$. Therefore, $\Tilde{\psi}_0$ descends to an automorphism $\psi_0$ on $H$. Denoting the corresponding automorphism of $JH$ by the same letter we have $$f^* \circ \psi_0 = \Tilde{\psi}_0 \circ f^* = f^*$$ implying $$f^*\circ (\Id_{JH} - \psi_0) = 0.$$ Since $\ker f^*$ is finite, we get $\psi_0 = \Id_{JH}$, hence $\psi_0 \in \Aut(H)$ is the identity. Therefore, $\Tilde{\psi}_0$ lies in $\Gal(\tC/H)$ hence it is enough to show that $\psi \mapsto \Tilde{\psi}_0$ is an injection. Given two elements $\varphi, \psi \in N$ inducing the same $\alpha \in \Aut(\tC)$ we see that either $\tilde{\varphi} = \tilde{\psi}$ or $\tilde{\varphi} =  (-1)_{J\tC}\circ \tilde{\psi}$. The latter case is impossible since $\tilde{\varphi}_{|f^*JH} = \tilde{\psi}_{|f^*JH} = \id_{f^*JH}$, hence $\tilde{\varphi} = \tilde{\psi}$ implying $\varphi = \psi$. 
\end{proof} 

\vspace{3mm}

The next step of the proof is to recover the twelve elliptic curves inside the Prym variety. The following two propositions deal with non-isotropic and isotropic cases, respectively.

\begin{prop}
\label{elnoniso}
    Assume $f:\tC \to H$ is non-isotropic.
    Then the restricted polarisation on $P_{\sigma}$ is equal to $3\Theta_{\sigma}$ with $(P_{\sigma},\Theta_{\sigma}) \simeq (E \times E, \cO(E \times \{0\} + \{0\} \times E + A))$, where $A$ denotes the antidiagonal. Moreover, the curves $E_{\sigma,j}, E_{\sigma,j\tau}, E_{\sigma,j\tau^2}$ are unique elliptic curves on $P_{\sigma}$, defining a divisor corresponding to a polarising line bundle $\Theta_\s$.
\end{prop}

\begin{proof}

    It follows from \cite[Theorem 1, page 67]{ries} that $(P(h_\s),\Theta_\s) \simeq (E_{\s,j} \times E_{\s, j}, L)$ with $$\varphi_L = \begin{pmatrix}
        2 & -1 \\ -1 & 2
    \end{pmatrix}.$$
    The first part of the proposition follows from \cite[Lemma 5.4]{BLiso}.
In particular, the curves are uniquely determined by the polarisation. Since $f$ is non-isotropic, the Prym variety $P_{\s}$ is the image of $P(h_{\s})$ by the embedding $k_{\sigma}^*$, hence the same holds for $P_{\sigma}$.  
\end{proof}

\begin{rem}
    One can also see the Proposition \ref{elnoniso} from the perspective of \cite{BSan}. Let $D$ be an effective divisor of arithmetic genus 4 on $P(h_{\s})$ defined as
    $E_{\sigma,j}\cup E_{\sigma,j\tau}\cup E_{\sigma,j\tau^2}$. Since $D$ is the sum of abelian subvarieties of $P(h_{\s})$, it is symmetric. Moreover, every elliptic curve passes through $0 \in P(h_{\s})$ and three $2$-torsion points, so $D$ passes through ten 2-torsion points with odd multiplicity. Combining \cite[Prop 3.3]{BSan} and \cite[Prop 2.5]{BSan}, we see that symmetric effective divisors on $P(h_{\s})$ pass through either six or ten 2-torsions, and there is a unique divisor that passes through ten 2-torsions. Therefore, these properties characterise $(1,3)$-theta divisors, so $D$ is (up to translation) a unique $(1,3)$-theta divisor as defined in \cite{BSan} (see also remark before Section 3.3 therein). As a consequence, the curves are uniquely determined by the polarisation.
\end{rem}

\begin{prop}
\label{eliso}
     Assume $f:\tC \to H$ is isotropic. Then the curves $E_{\sigma,j}, E_{\sigma,j\tau}, E_{\sigma,j\tau^2}$ being images of (distinguished) elliptic curves from $P(h_\s)$ by $k_\s^*$ can be intrinsically distinguished from $(P,\Theta_P)$. 
\end{prop}
\begin{proof}
Note that by Lemma \ref{trivialkernel}, using its notation, we have  $$h_{\sigma}^*JH \cap P(h_{\sigma}) = h_{\sigma}^*((\ker h_{\sigma}^*)^{\bot})=\left<\delta,
h_{\s}^*(\zeta)\right>,$$ for some $\delta\in h_{\sigma}^*JH[3]$.
Since $f$ is isotropic, Lemma \ref{trivialkernel} implies that $h_{\s}^*(\zeta) \in \ker k_{\s}^*$, so $P_\s \cap f^*JH=\left<\delta\right> \subset J\tC$ is isomorphic to $\ZZ_3$. On the other hand, $K(\Theta_P)=P\cap f^*JH\cong (\ZZ_3)^2$.
Combining these together, we conclude that $K(\Theta_P)\cap P_\s=P_\s \cap f^*JH=\left<\delta\right>\cong\ZZ_3$.

Now, using results from the non-isotropic case, we know that $P(h_{\sigma})$ with the restricted polarisation $\Xi$ is a $(1,3)$ polarised abelian surface with three distinguished elliptic curves. Note that its dual $\widehat{P(h_{\sigma})}=P(h_{\sigma})/K(\Xi)=P_\s/\langle \delta \rangle$. Hence, starting from $P_\s$ and $\delta$, we find the dual to $P_\s/\delta$ with the unique map to $P_\s$ (that is dual to the quotient by $\delta$) and $E_{\sigma,j}, E_{\sigma,j\tau}, E_{\sigma,j\tau^2}$ are uniquely defined as images of distinguished curves from $P(h_{\sigma})$. 
\end{proof}

In the following remark, we will see a different way of distinguishing elliptic curves for a general isotropic Prym. 

\begin{rem}
Recall that $P(h_{\sigma})\subset JC_\s$ is a $(1,3)$ polarised surface with the polarisation defined in Proposition \ref{elnoniso}. If we take the curve $H$ to be general, we can assume that the elliptic curve $E$ has $j$-invariant different from $0$. Now, since we are in the isotropic case,  $P_\s$ is isomorphic to $P(h_{\sigma})/\langle h_\s^*(\zeta) \rangle$. Using \cite[Prop 3.6]{BLExE}, we get that $P_\s=(J(X),3\Theta)$ is isomorphic to a Jacobian of a genus 2 curve $X$ with triple principal polarisation. Moreover, note that $\tau$ is an automorphism of order $3$ on $JC_\s$ fixing $h_\s^*(\zeta)$, hence it is also an automorphism on $JX$. Similarly, $j$ (viewed as a lift of the hyperelliptic involution of $H$ to $C_\s$) descends to an involution on $X$. In this way we conclude that $X$ is Bolza's case III curve (see \cite[p.340]{BL} or \cite{Bolza}).  
Thus, a general $X$ has an automorphism group isomorphic to $D_6$. Now, we can see that $X$ has 7 involutions, namely, the hyperelliptic involution $\iota$ and 3 pairs of involutions ($j,j\iota)$. We saw that each involution on $P_\s = JX$ either fixes $\delta$ or acts as $(-1)$. We denote by $j, j\tau, j\tau^2$ precisely three (non-hyperelliptic) involutions that acts on $\delta$ as $(-1)$ and by $E_{\sigma,j}, E_{\sigma,j\tau}, E_{\sigma,j\tau^2}$ we denote their fixed points set (that are elliptic curves since the involutions acts on a genus 2 curve, the quotient curves have to have genus 1).  Hence, we are able to distinguish three elliptic curves.

Note that the isotropic case is more difficult due to the following fact. In the non-isotropic case, one can easily distinguish between involutions $j$ and $-j$ by computing the restricted polarisation types of their fixed point sets. In particular, one of them has type $2$ while the other has type $6$. However, in the isotropic case, since we have to divide by a 3-torsion point, both restricted polarisation types are $2$. This is the reason why in the proof one has to consider the quotient surface and its dual together with the results from the non-isotropic case.
\end{rem}

\vspace{3mm}

Now, we would like to show a correspondence between subsets of elliptic curves inside the Prym variety and involutions.
Note that the involutions $j\sigma^a\tau^b$ of $J\tC$ induce involutions on $P(f)$ and on each small Prym, which we denote by the same letters. We denote the set of nine involutions by 
$$M=\{j\sigma^a\tau^b|_{P(f)}: a,b\in\{0,1,2\}\}$$
We want to show that $M$ can be recovered from $P(f)$. 
Denote by $$\cE=\{E_{\sigma,j}, E_{\sigma,j\tau}, E_{\sigma,j\tau^2},\ E_{\tau,j}, E_{\tau,j\s}, E_{\tau,j\s^2}, \ E_{\st,j}, E_{\st,j\tau}, E_{\st,j\tau^2}, \ E_{\st^2,j}, E_{\st^2,j\tau}, E_{\st^2,j\tau^2}\}$$
the set of distinguished elliptic curves, which can be recovered using the Propositions \ref{elnoniso} and \ref{eliso}.

We will focus on the curve $E_{\s,j}$; analogous results follow for other elliptic curves by symmetry. Note that the symmetric idempotent of $E_{\sigma, j}$ in $P(f)$ is $$\varepsilon_{\sigma,j} = \frac{(1+j)(1+\sigma + \sigma^2)}{6}.$$ Therefore, the involutions $j, j\sigma$ and $j\sigma^2$ act as the identity on $E_{\sigma,j}$ and as $(-1)$ on its complementary abelian subvariety $E_{\sigma, j}^c$ in $P_{\sigma}$. Analogously, for any other elliptic curve $E\in\cE$ there are 3 involutions from $M$ that act as the identity on it. Moreover, each involution acts as the identity on precisely 4 elliptic curves, each one in every small Prym. Lastly, by simple counting, taking any two small Pryms and an elliptic curve inside each of them, there exists a unique involution from $M$ that acts as the identity on the chosen elliptic curves. To summarise, we have the following correspondences that we will use frequently in the proofs.

\begin{correspondence}
\label{correspondence}
\begin{align*}
    M &\ \ \ \ \ \ \cE\\
    \text{an involution} &\longrightarrow \text{identity on 1 elliptic curve in each small Prym}\\
    \text{3 involutions that are identity on it} &\longleftarrow \text{an elliptic curve}\\
    \text{a unique involution that is identity on them} &\longleftarrow \text{any 2 elliptic curves (not in the same triple)}
    \end{align*}
\end{correspondence}
Lastly, taking an involution, say $j\tau$, and a chosen triple in a small Prym, say $E_{\sigma,j}, E_{\sigma,j\tau}, E_{\sigma,j\tau^2}$, the involution acts as an identity on one of the curves and we can prove that it swaps the other curves. Since $j\tau(1+j\tau)=(1+j\tau)$, we get that $j\tau(\varepsilon_{{\s, j\tau}})=\varepsilon_{{\s,j\tau}}$. Let us perform a necessary computation to check the swapping:
$$j\tau(6\varepsilon_{{\s,j}}) = (j\tau + \tau^2)(1+\s+\s^2)=(j\tau^2+1)\tau^2(1+\s+\s^2)=(1+j\tau^2)(1+\s+\s^2)(\tau^2)=(6\varepsilon_{{\s,j\tau^2}})(\tau^2)$$

Now, we are ready to make the following definition.
\begin{defn}
    We say that an involution $q \in \Aut(P(f))$ is $\langle\sigma\rangle$-compatible if there exist $a,b$ such that $q|_{P_{\sigma}} = j\sigma^a\tau^b|_{P_{\sigma}}$. Analogously, we define $\langle\tau\rangle$-compatibility, $\langle\sigma\tau\rangle$-compatibility and $\langle\sigma^2\tau\rangle$-compatibility by taking other small Pryms.
    
    We say that an involution $q \in \Aut(P(f))$ is group-compatible if it is $\langle\sigma\rangle$-compatible, $\langle\tau\rangle$-compatible, $\langle\sigma\tau\rangle$-compatible and $\langle\sigma^2\tau\rangle$-compatible.
\end{defn}

Note that an involution $q$ is $\langle \sigma \rangle$-compatible if and only if it acts as the identity on either $E_{\sigma,j}, E_{\sigma,j\tau}$ or $E_{\sigma,j\tau^2}$ and as $(-1)$ on the corresponding complementary curve.
It follows from the definition that all the involutions in $M$ are group-compatible. Below, we will show that these are the only involutions on $P(f)$ with this property.

\begin{prop}
\label{global}
    Let $q$ be a group-compatible involution on $P(f)$. Then $q \in M$.
\end{prop}
    
\begin{proof} 
    
    Let $q \in \Aut(P(f))$ be a group-compatible involution. 
    Since $q$ is $\langle \sigma \rangle$-compatible and $\langle \tau \rangle$-compatible without loss of generality we can assume  that $q$ acts as the identity on $E_{\sigma, j} \subset P_{\sigma}$ and $E_{\tau, j} \subset P_{\tau}$. In particular $q=j$ on $P_{\sigma}+P_{\tau}$ (where $+$ denotes an algebraic sum of abelian subvarieties).

   By the correspondence \ref{correspondence}, for any two elliptic curves in  $P_{\sigma\tau}$ and $P_{\sigma^2\tau}$ there is a unique involution $j\sigma^a\tau^b$ which acts as the identity on both curves. Hence $q=j\s^a\tau^b$ on $P_{\sigma\tau}+P_{\s^2\tau}$ for some $a,b$.
   
      Thus, from the isotypical decomposition of $P(f)$ (Corollary \ref{decprym}) we have the following equality for $x_1 \in P_{\sigma},\ x_2 \in P_{\tau},\ x_3 \in P_{\sigma\tau},\ x_4 \in P_{\sigma^2\tau}$: $$q(x_1 + x_2 + x_3 + x_4) = j(x_1 + x_2) + j\sigma^a\tau^b(x_3 + x_4).$$ 
      
      Assume that $q \notin M$, hence $\sigma^a\tau^b \neq \id$. We will show that this assumption contradicts Lemma \ref{degree} by bounding the cardinality of the kernel of the addition map $\mu$ from above. Take a tuple $(x_1, x_2, x_3, x_4) \in P_{\sigma} \times P_{\tau} \times P_{\sigma\tau}\times P_{\sigma^2\tau}$ such that $x_1 + x_2 + x_3 + x_4 = 0$. Then $0 = j(x_1 + x_2) + j\sigma^a\tau^b(-x_1 - x_2),$ so that $x_1 + x_2 = \sigma^a\tau^b(x_1 + x_2)$. Analogously, $x_3 + x_4 = \sigma^a\tau^b(x_3 + x_4)$.
    Since one of the small Pryms is fixed by $\sigma^a\tau^b$, without loss of generality   we can assume $\sigma^a\tau^b = \sigma$, which implies that $\sigma(x_2) = x_2$.
    Therefore, for any $(x_1, x_2, x_3, x_4)$ in the kernel of the addition map $$\mu: P_{\sigma} \times P_{\tau} \times P_{\sigma\tau}\times P_{\sigma^2\tau} \to P(f)$$ we have $x_2 \in \Fix(\sigma, \tau)$.  Note that $x_3 = \sigma(\sigma^2(x_3))$ and, since $\sigma^2(x_3) \in P_{\tau}$ we have $\sigma^2(x_3) \in \Fix(\s)$ which implies that $x_3 \in \Fix(\sigma)$. Similarly, we get $\{x_1, x_2, x_3, x_4\} \in \Fix(\sigma, \tau)$. Now, $$P_{\sigma} \cap \Fix(\sigma, \tau) = P_{\sigma} \cap f^*JH = k_{\sigma}^*P(h_{\sigma}) \cap k_{\s}^*h_{\s}^*JH = k_{\s}^*(P(h_\s) \cap h_{\s}^*JH),$$ where the last equality follows from the fact that $\ker k_{\sigma}^{*}|_{P(h_{\s})} \subset P(h_\s) \cap h_{\s}^*JH$ in both isotropic and non-isotropic cases. Note that $P(h_\s) \cap h_{\s}^*JH \simeq \mathbb{Z}_3 \times \mathbb{Z}_3$ since the restricted polarisation from $JC_{\sigma}$ to $P(h_{\sigma})$ is of type $(1,3)$. It follows from Lemma \ref{kernel} that in the non-isotropic case $P_{\sigma} \cap \Fix(\sigma, \tau) \simeq \mathbb{Z}_3 \times \mathbb{Z}_3$ and in the isotropic case $P_{\sigma} \cap \Fix(\sigma, \tau) \simeq \mathbb{Z}_3$. 
    
    It follows from the above that in the non-isotropic (resp. isotropic) case the cardinality of $\ker \mu$ is at most $3^8$ (resp. $3^4$), which contradicts Lemma \ref{degree}. Therefore, $q$ must be in $M$.

    \end{proof}

    \begin{defn}
        We say that a 4-tuple of elliptic curves $(E_1, E_2, E_3, E_4)$ with $E_1 \in P_{\sigma},\ E_2 \in P_{\tau},\ E_3 \in P_{\sigma\tau},\ E_4 \in P_{\sigma^2\tau}$ is compatible if there is a group-compatible involution $q$ (in other words, $q \in M$) that acts as the identity on $E_i$ for each $i$.
    \end{defn}
    For instance, the 4-tuple ($E_{\sigma, j}, E_{\tau, j}, E_{\sigma\tau, j}, E_{\sigma\tau^2,j}$) is compatible. 
    \begin{rem}
    \label{compatible4}
        There are precisely nine compatible 4-tuples. It follows from Proposition \ref{rozklad} that each compatible 4-tuple generates the Jacobian $JC_{ab}$ inside $P(f)$ for some $a,b$ and there is a one-to-one correspondence between compatible 4-tuples and Jacobians $JC_{ab}$ inside $P(f)$.
    \end{rem}

\vspace{3mm}

We are finally ready to give the proof of the main theorem.

\begin{thm}\label{mainthm}
    The Prym maps $$\mathcal{P}^{iso}:(\mathcal{R}_2^{\mathbb{Z}_3\times \mathbb{Z}_3})^{iso} \to \cA^{(1,1,1,1,1,1,3,3)}_8$$ 
    $$\mathcal{P}^{ni}:(\mathcal{R}_2^{\mathbb{Z}_3\times \mathbb{Z}_3})^{ni} \to\cA^{(1,1,1,1,1,1,1,9)}_8$$
    are injective.
\end{thm}

\begin{proof}
    
Let $$\mathcal{P}:(\mathcal{R}_2^{\mathbb{Z}_3\times \mathbb{Z}_3}) \to \cA^{(1,1,1,1,1,1,3,3)}_8 \cup \cA^{(1,1,1,1,1,1,1,9)}_8$$ be the global Prym map and $(P,\Theta_P)$ be an element of $\im\mathcal{P}$.
Fix a covering $f: \tC \to H$ with $P(f) \cong P$. By Proposition \ref{group} we can recover the action of $\langle \sigma, \tau \rangle \simeq \mathbb{Z}_3 \times \mathbb{Z}_3$ on $P$ induced by $\Gal(\tC/H)$ from the polarisation $\Theta_P$. According to the proof of Proposition \ref{rozklad} the isotypical decomposition of $P$ is as follows:
$$P = P_{\sigma} \boxplus P_{\tau} \boxplus P_{\sigma\tau} \boxplus P_{\sigma^2\tau},$$
where $P_{\sigma}$ is a subvariety of $P$ with a symmetric idempotent $$\varepsilon_{P_{\sigma}} = \frac{(1 + \sigma + \sigma^2)(2 - \tau - \tau^2)}{9},$$ and analogously for other components. Our aim is to show that we can recover $f$ uniquely from $P$.

It follows from Proposition \ref{elnoniso} in the non-isotropic case and Proposition \ref{eliso} in the isotropic case that the restricted polarisation uniquely determines three elliptic curves inside each isotypical component of $P$. We denote the elliptic curves inside $P_{\sigma}$ by $E_{\sigma}^1, E_{\sigma}^2$ and $E_{\sigma}^3$. Now, define involutions $j_{\sigma}^i$ on $P_{\sigma}$ given by the identity on $E_{\sigma}^i$ and $(-1)$ on its complementary abelian subvariety $(E_{\sigma}^1)^c$ inside $P_{\sigma}$. We define analogous involutions on other isotypical components.

It follows from Proposition \ref{global} that there exists an involution on $P$ that restricts to the involutions $j_{\sigma}^{i_1}, j_{\tau}^{i_2}, j_{\sigma\tau}^{i_3}, j_{\sigma^2\tau}^{i_4}$ on the corresponding components if and only if $(E_{\sigma}^{i_1}, E_{\tau}^{i_2}, E_{\sigma\tau}^{i_3}, E_{\sigma^2\tau}^{i_4})$ is a compatible 4-tuple for the covering $f$. Therefore, we can recover nine compatible 4-tuples for $f$.
Moreover, Proposition \ref{global} also shows that compatible 4-tuples are the same for any coverings $f_1$ and $f_2$ with $P(f_1) \cong P(f_2) \cong P$. Therefore, it remains to show that compatible 4-tuples determine the covering uniquely.  

Fix a compatible 4-tuple $(E_1, E_2, E_3, E_4)$ and consider an abelian subvariety $J$ of $P$ generated by $E_1, E_2, E_3, E_4$. Note that the restricted polarisation $\Theta_P|_J$ is twice the principal polarisation, and it follows from Remark \ref{compatible4} that $J$ is a Jacobian of a genus 4 curve $C$ corresponding to one of the curves $C_{kl}$. The embedding of $E_i$ in $JC$, yields a covering $C\to E_i$ and by taking its Galois closure, we recover the curve $\tC$ with the $S_3$ deck group of automorphisms. By taking all $i$, we see that automorphisms generate SmallGroup(18,4) and there is a unique $\mathbb{Z}_3 \times \mathbb{Z}_3$ subgroup such that the quotient $f:\tC \to H$ has the property $P(f) = P$. Note that all nine compatible 4-tuples give the same covering $f$ by Remark \ref{compatible4}, hence the Prym map is injective.
\end{proof}
\begin{rem}
    Our proof is not constructive in the sense that we use the inverse of the Torelli map to find genus $4$ curves $C_{ab}$.
\end{rem}

Finally, using results from \cite{BLiso}, we can find a birational model of $(\cR_2^{\ZZ_3\times\ZZ_3})^{ni}$ with the following proposition.
\begin{prop}
    The moduli $(\cR_2^{\ZZ_3\times\ZZ_3})^{ni}$ can be embedded as an open dense subset of the moduli of $(1,9)$-polarised abelian surfaces $\cA_2^{{1,9}}$.
\end{prop}
\begin{proof}
Note that we have shown that for a $\ZZ_3\times\ZZ_3$ covering $f$, we get that $f^*JH=JH/\left<\eta,\zeta\right>$ is a $(1,9)$ polarised abelian surface embedded in $J\tC$. One can consider a dual map to the embedding, called $\varphi:J\tC\to\widehat{f^*JH}$ and its restriction to $\tC\subset J\tC$. By \cite[Prop 4.3]{BLiso}, we see that $\varphi(\tC)\subset \widehat{f^*JH}$ defines the dual $(1,9)$ polarisation and therefore, being of arithmetic genus 10, is isomorphic to $\tC$. Now, the covering $f:\varphi(\tC)\to H$ can be seen as a quotient of $\widehat{f^*JH}$ by a unique (isotropic) $\ZZ_3\times \ZZ_3$ subgroup inside $K(\Theta_{\widehat{f^*JH}})\simeq\ZZ_9\times\ZZ_9$. 

On the other hand, starting from an abelian surface $A\in\cA_2^{{1,9}}$, we can divide $A$ by a unique isotropic $\ZZ_3\times\ZZ_3$ subgroup and if we get a Jacobian $JH$ of a smooth curve downstairs (which we do for a dense open subset of $\cA_2^{{1,9}}$), we take $\tC$ to be preimage of $H$ to get a $\ZZ_3\times\ZZ_3$ covering $f:\tC\to H$. Using again \cite[Prop 4.3]{BLiso}, we get $f^*JH=\widehat{A}$, so these constructions are inverse to each other.

The existence of $\ZZ_3\times\ZZ_3$ coverings on $(1,9)$-polarised abelian surfaces has also been explained in the recent preprint \cite[Thm 4.3]{HS}.
\end{proof}

\section{Prym map of abelian unramified coverings of genus two curves}

The purpose of this section is to study finiteness of the Prym map associated to abelian Galois unramified coverings of genus two curves. We will show that the Prym maps are generically finite with only a few exceptions. We start by briefly recalling the setup of abelian  Galois unramified coverings following \cite{Pardini, Mohajer}. 

Let $f: \tC \to H$ be an unramified Galois covering of a smooth genus 2 curve $H$ with an abelian Galois group $G$. Then \(f_*\mathcal{O}_{\tC} = \oplus_{\chi \in G^*} L_{\chi}^{-1}\), where \(L_{\chi}\) is an invertible sheaf on $H$ on which $G$ acts via the character \(\chi\). The algebra structure of $f_*\mathcal{O}_{\tC}$ is compatible with the \(G\)-action, so the multiplication is given by \(\mathcal{O}_{\tC}\)-linear maps $$\mu_{\chi, \chi'}: L_{\chi}^{-1} \otimes L_{\chi'}^{-1} \to L_{\chi\chi'}^{-1}$$ for all \(\chi, \chi' \in G^*\). For unramified abelian coverings, the data \((\{L_{\chi}\}_{\chi \in G^*})\) completely determines the covering \cite[Theorem 2.1]{Pardini}. Therefore (see also \cite[Proposition 2.3]{Mohajer}), we can define the moduli space \(\mathcal{R}_2^G\) of Galois unramified coverings of genus 2 curves with an abelian Galois group $G$ as the moduli space of objects \((H, \{L_{\chi}\}_{\chi \in G^*})\) with line bundles $L_{\chi}$ as above.

Recall that Prym varieties $P(f)$ of coverings $f: \tC \to H$ with a non-cyclic group $G$ may have different polarisation types. Therefore, both domain \(\mathcal{R}_2^G\) and co-domain $\mathcal{A}^G$ of the Prym map $\mathcal{P}_2^G$ may be reducible.

We begin the study of generic finiteness of the Prym map with the following result.

\begin{prop}
\label{smallgroups}
    Let \(f: \tC \to H\) be an unramified Galois covering of a smooth genus 2 curve $H$ with an abelian Galois group \(G\). If $G$ is cyclic of order at least $6$ or $G$ is isomorphic to either $\mathbb{Z}_2 \times \mathbb{Z}_2$, $\mathbb{Z}_3 \times \mathbb{Z}_3$ or $\mathbb{Z}_5 \times \mathbb{Z}_5$, then the Prym map $$\mathcal{P}_2^{G}: \mathcal{R}_2^{G} \to \mathcal{A}^{G}$$ is generically finite. 
\end{prop}

\begin{proof}

    For the proof in the cyclic case see \cite{Ago} and for the proof in the case \(G \simeq \mathbb{Z}_2 \times \mathbb{Z}_2\) see \cite{BO19}. Moreover, we have already proved injectivity in the $\ZZ_3\times\ZZ_3$ case. It is left to prove the statement in the case $G = \ZZ_5 \times \ZZ_5$, but, since our proof works in a more general setting, we present it below in full generality.

    Let $G \simeq \mathbb{Z}_p \times \mathbb{Z}_p = \langle \sigma, \tau \rangle$ with an odd prime $p$.     
    It follows from \cite{ries} that the hyperelliptic involution $\iota \in \Aut(H)$ lifts to an involution $j \in \Aut(\Tilde{C})$ so that $\sigma, \tau, j$ generate the subgroup $F \subseteq \Aut(\Tilde{C})$,  given by $$F = \langle \sigma, \tau, j \: | \: \sigma^p = \tau^p = j^2 = (j\sigma)^2 = (j\tau)^2 = 1, \sigma\tau = \tau\sigma \rangle.$$ As in Section 2, we consider the quotients $\pi_{ab}: \tC \to C_{ab} := \tC/\langle j\sigma^a \tau^b\rangle$ for $0 \leq a,b \leq p-1$.
    Since all involutions in $F$ are conjugate (the proof is analogous to the case $p=3$ in Lemma \ref{conj}), $C_{ab}$ are pairwise isomorphic curves of genus $\frac{p^2-1}{2}$, see \cite{ortega}. Therefore, there is a well-defined map $$\psi: \mathcal{R}_{2}^{\mathbb{Z}_p \times \mathbb{Z}_p} \ni (H, \eta, \zeta) \mapsto [C_{00}] \in \mathcal{M}_{\frac{p^2-1}{2}}.$$ 
    Moreover, composing $\psi$ with the Torelli map $t \times t$ we obtain a map $$T: \mathcal{R}_2^{\mathbb{Z}_p \times \mathbb{Z}_p} \ni (H, \eta, \zeta) \mapsto JC_{00} \times JC_{10} \in \mathcal{A}_{p^2-1}.$$ By repeating the reasoning from Remark \ref{prymisog} we can see that the abelian varieties $P$ and $JC_{00} \times JC_{10}$ are isogenous. Assume that $\psi$ has finite fibers. 
    Let $P \in \mathcal{A}^{\mathbb{Z}_p \times \mathbb{Z}_p}$ and $$T_P = \{\alpha \in \mathcal{R}_2^{\mathbb{Z}_p \times \mathbb{Z}_p} \: | \: \text{there exists an isogeny } f:T(\alpha) \to P\}.$$ Note that $(\mathcal{P}_2^{\mathbb{Z}_p \times \mathbb{Z}_p})^{-1}(P) \subseteq T_P$.

    We claim that there are at most countably many curves $C$ (up to isomorphism) such that there is an isogeny $\beta: JC \times JC \to P$. Indeed, $P$ has countably many abelian subvarieties, and the images under $\beta$ of $JC \times \{0\} \subset JC \times JC$ are isogenous to $JC$. Since there is at most countably many abelian varieties isogenous to $JC$, and the Torelli map is injective, the claim is proved.  

It follows from the claim and finiteness of $\psi$ that $T_P$ is at most countable. Therefore, $(\mathcal{P}_2^{\mathbb{Z}_p \times \mathbb{Z}_p})^{-1}(P)$ is zero-dimensional, hence it is finite.

Now let us show that $\psi$ indeed has finite fibers. We claim that the map $$\gamma: \Tilde{C} \to C_{00} \times C_{01}$$ given by $$\gamma(x) = (\pi_{00}(x), \pi_{10}(x))$$ is an embedding. First we show that that $\gamma$ is injective. If $x_1 \neq x_2 \in \Tilde{C}$ with $\gamma(x_1) = \gamma(x_2)$ then $x_1 = j(x_2)$ and $x_1 = j\sigma(x_2)$ implying that $x_2$ is a fixed point of $\sigma$, which is impossible. 
On the other hand, the map $\gamma$ is not smooth at $x$ if and only if $x$ is a fixed point of both $j$ and $j\sigma$, which is again impossible. Therefore, $\gamma$ is an embedding.

Now, we compute the self-intersection of $\gamma(C)$ inside $C_{00} \times C_{10}$. We claim that 
$$\gamma(C)^2 = 12 - 2p^2.$$
Indeed, it follows from the adjunction formula that $$\gamma(C)^2 = 2p^2 - \gamma(C).K_{C_{00} \times C_{10}}.$$ Since the degree of $\pi_{kl}$ is equal to 2, we have $$\gamma(C).(C_{00} \times \{q\}) = \gamma(C).(\{q\} \times C_{10}) = 2.$$ Therefore, $$\gamma(C).K_{C_{00} \times C_{10}} = 2 \cdot (2g(C_{00}) - 2 + 2g(C_{10}) - 2) = 4p^2 - 12,$$ and the result follows.

Now, assume that the map $\psi$ has positive dimensional fibres. Then $\gamma(C)$ must move in an algebraic family. Therefore, $\gamma(C)^2 \geq 0$, which is possible only for $p^2 \leq 6$. Hence, for all odd primes $p$ the fibres are finite, which finishes the proof.
\end{proof}

Now we are ready to prove the main result of this section. 
\begin{thm}
\label{thmabel}
    Let $G$ be a finite abelian group.
    \begin{enumerate}
        \item If $G$ is isomorphic to either $\mathbb{Z}_2, \mathbb{Z}_3, \mathbb{Z}_4$ or $\mathbb{Z}_5$, then the generic fiber of the Prym map $$\mathcal{P}_2^{G}: \mathcal{R}_2^{G} \to \mathcal{A}^{G}$$ is positive dimensional. 
        \item If $G$ is not isomorphic to a cyclic group of order less than 6, then the Prym map $$\mathcal{P}_2^{G}: \mathcal{R}_2^{G} \to \mathcal{A}^{G}$$ is generically finite. 
    \end{enumerate}
    Note that $\mathcal{R}_2^{G}$ is non-empty if and only if the minimal number of generators of $G$ is at most 4 and in particular if and only if there is an embedding of $G$ into a Jacobian of a genus $2$ curve.
\end{thm}

\begin{proof}
    For the dimensions of the generic fiber of the Prym map for \(G\) isomorphic to either \(\mathbb{Z}_3\) or \(\mathbb{Z}_5\) see \cite{Xiao}. For \(G \simeq \mathbb{Z}_2\), the Prym variety is an elliptic curve, so the generic fiber of $\mathcal{P}_2^{\mathbb{Z}_2}$ is two-dimensional.
    
The generic fiber of $\mathcal{P}_2^{\mathbb{Z}_4}$ is one-dimensional by \cite{anatoli}.

    Now, assume that $G$ is not isomorphic to a cyclic group of order less than six. It follows that $G$ contains a subgroup $B$, which is either cyclic of order at least six or isomorphic to one of the groups $\mathbb{Z}_2 \times \mathbb{Z}_2$, $\mathbb{Z}_3 \times \mathbb{Z}_3$ or $\mathbb{Z}_5 \times \mathbb{Z}_5$. 

The idea of the proof is to show that from the Prym map one can recover (up to a finite choice) the Prym map of an intermediate covering, hence generic finiteness follows.
To make this precise, firstly fix the groups $B$ and $G$ and an irreducible component of $\mathcal{R}^G$. Let us define
$$X = \{((H,G), (H,B)) \in \mathcal{R}^G\times\mathcal{R}^B: g(H) = 2, B \subset G \subset JH\}\subset \mathcal{R}^G\times\mathcal{R}^B,$$ 
where we abuse notation by writing $(H,G)$ for $(H, \{L_{\chi}\}_{\chi \in G^*}) \in \mathcal{R}^G$ and analogously for $\mathcal{R}^B$.

Similarly, let us define
$$Y=\{(P,Q)\in \mathcal{A}^G\times\mathcal{A}^B: P\in\im\mathcal{P}^G, Q\in\im\mathcal{P}^B: \exists \: \iota: Q \to P \text{ coming from Prym construction}\}
$$
where \textit{coming from Prym construction} means that there exists a tower of coverings $\tC\xrightarrow{h}C\xrightarrow{g} H$ with $[g\circ h: \tC \to H] \in \mathcal{R}^G$ and $[g: C \to H] \in \mathcal{R}^B$ such that $P=P(\tC/H),\ Q=P(C/H)$ and $\iota=h^*|_{P(C/H)}$.

Let us consider the commutative Diagram \ref{diagram4}.
    \begin{figure}
    \begin{tikzcd}[row sep={60,between origins}, column sep={100,between origins}]
X \arrow[d, "pr_1"] \arrow[r, "\mathcal{P}^G\times\mathcal{P}^B"] & Y \arrow[d, "pr^Y_1"] \\
\mathcal{R}^G \arrow[r, "\mathcal{P}^G"]                    & \im\mathcal{P}^G\subset\mathcal{A}^G
\end{tikzcd}
\caption{}
\label{diagram4}
\end{figure}
Note that both projections are surjective and finite maps. The map $pr_1$ is finite since there is a finite number of subgroups of $G$ isomorphic to $B$.

To prove the finiteness of $pr_1^Y$, let $P \in \im \mathcal{P}^G$. Note that any $\iota: Q \to P$ coming from Prym construction has fixed degree that depends only on $B$ and $G$, hence the type of $\iota(Q)$ is bounded from above. Moreover, there exist only finitely many abelian subvarieties of $P$ of bounded restricted polarisation type, see \cite{BLfinite}. Finally, given a subvariety $A \subset P$, there are only finitely many possible maps $\iota: Q \to P$ coming from Prym construction with $\iota(Q) = A$ since there are finitely many dual maps $\widehat{\iota}: \widehat{A} \to \widehat{Q}$ as each $\widehat{\iota}$ is the quotient of $\widehat{A}$ by a finite subgroup of bounded cardinality. Therefore, $(pr_1^Y)^{-1}(P)$ is finite.

Now, $\dim(X)=3$ because it is a finite covering of $\mathcal{R}^G$. The next step is to show that $Y$ is also of dimension $3$. Certainly, $\dim(Y)\leq 3$ because $pr^Y_1$ is finite and $\im\mathcal{P}^G$ is the image of a $3$-dimensional space. On the other hand, the projection $pr^Y_2$ to $\im\mathcal{P}^B$ is surjective (to a $3$-dimensional space), hence $Y$ is of dimension 3.

Take a general $(P,Q)\in Y$. Then $Q$ is general, hence $(\mathcal{P}^B)^{-1}(Q)$ is finite by Proposition \ref{smallgroups}. Therefore, we have a finite number of pairs $(H,B)\in\mathcal{R}^B$ that can be extended to a finite number of $(H,G)\in\mathcal{R}^G$ as $JH$ has finitely many torsion points of order at most $|G|$. In this way we get a finite number of $(H,B,G)$ that are possible preimages of $(P,Q)$, so the top arrow is generically finite. 

Since the diagram commutes and all arrows apart from $\mathcal{P}^G$ are generically finite, we get that $\mathcal{P}^G$ is generically finite, too.

        \end{proof}

\bibliographystyle{alpha}
\bibliography{biblio}

\textsc{P. Bor\'owka, Institute of Mathematics, Jagiellonian University in Krak\'ow, Poland}\\
\textit{email address:} pawel.borowka@uj.edu.pl

\textsc{A. Shatsila, Doctoral School of Exact and Natural Sciences, Jagiellonian University, Poland}\\
\textit{email address:} anatoli.shatsila@doctoral.uj.edu.pl.

\end{document}